\documentclass[reqno,10pt]{amsart}

\usepackage{amsmath,amssymb,amsthm,mathrsfs}
\usepackage{enumerate}
\usepackage{graphicx}
\usepackage{wasysym}
\usepackage{color}

\allowdisplaybreaks

\theoremstyle{plain}

\newtheorem{theorem}{Theorem}[section]
\newtheorem{lemma}[theorem]{Lemma}
\newtheorem{proposition}[theorem]{Proposition}
\newtheorem{corollary}[theorem]{Corollary}
\newtheorem{definition}[theorem]{Definition}

\numberwithin{equation}{section}

\begin{document}

\title[Two variable Freud OP and matrix Painlev\'e-type difference equation]{Two variable Freud orthogonal polynomials and matrix Painlev\'e-type difference equations}

\author[C. F. Bracciali, G. S. Costa,  T. E. P\'{e}rez]
{Cleonice F. Bracciali, Glalco S. Costa, Teresa E. P\'{e}rez}

\thanks{[C.~F.~Bracciali] Departamento de Matem\'{a}tica, IBILCE, UNESP - Universidade Estadual Paulista,
15054-000, S\~ao Jos\'e do Rio Preto, SP, Brazil. 
E-mail: cleonice.bracciali@unesp.br}

\thanks{[G.~S.~Costa] Departamento de Matem\'{a}tica, Instituto de Ci\^encias Tecnol\'ogicas e Exatas 
- ICTE, Universidade Federal do Tri\^{a}ngulo Mineiro - UFTM,
38025-180, Uberaba, MG, Brazil. 
E-mail: glalco.costa@uftm.edu.br}

\thanks{[T.~E.~P\'{e}rez] Instituto de Matem\'{a}ticas IMAG \&
Departamento de Matem\'{a}tica Aplicada, Facultad de Ciencias. Universidad de Granada. 18071. Granada, Spain. 
E-mail: tperez@ugr.es}

\thanks{%
This work was supported through the Brazilian Federal Agency for Support and Evaluation of Graduate Education (CAPES), in the scope of the Program CAPES-PrInt, process number 88887.310463/2018-00, International Cooperation Project number 88887.575407/2020-00}

\thanks{%
Third author (TEP) thanks FEDER/Junta de Andaluc\'ia A-FQM-246-UGR20; PGC2018-094932-B-I00 supported by MCIN/AEI 10.13039/501100011033 and FEDER funds, and IMAG-Mar\'ia de Maeztu grant CEX2020-001105-M}

\date{\today}

\begin{abstract}

We study bivariate orthogonal polynomials associated with Freud weight functions depending on real parameters. We analyse relations between the matrix coefficients of the three term relations for the orthonormal polynomials as well as the coefficients of the structure relations satisfied by these bivariate semiclassical orthogonal polynomials, also a matrix differential-difference equation for the bivariate orthogonal polynomials is deduced. The extension of the Painlev\'e equation for the coefficients of the three term relations to the bivariate case and a two dimensional version of the Langmuir lattice are obtained.

\end{abstract}

\subjclass[2010]{Primary: 42C05; 33C50}

\keywords{Bivariate orthogonal polynomials, 
Freud orthogonal polynomials, 
Three term relations, 
Matrix Painlev\'e-type difference equations}

\maketitle

\section{Introduction} 

The study of orthogonal polynomials with respect to the generalized weight function $|x|^\rho $ $\exp(-|x|^m)$, $\rho >-1, $ $m>0$, began with G\'{e}za Freud, see for example \cite{Fr76}. We refer to \cite{CJ18} for a interesting historic summary about the studies of generalized Freud polynomials.  

A symmetric Freud weight function in one variable is usually given by
\begin{equation*} 
w_{t}(x)=e^{-x^4+tx^2}, 
\end{equation*}
for $x \in \mathbb{R}$, and $t\in \mathbb{R}$ is consider as a time parameter. The corresponding moments exist and depend on $t$ as
\begin{equation*}
\mu_k(t) = \int_{-\infty}^{+\infty} x^k e^{-x^4 + tx^2} dx, \quad k=0,1,\ldots .
\end{equation*}

Therefore, the sequence of orthonormal polynomials with respect to $w_{t}(x)$ is a sequence of polynomials on the variable $x$ whose coefficients depend on $t$, that we denote $\{p_n(x,t)\}_{n\geqslant 0}$, and satisfies the three term recurrence relation in the form
$$
 x\,p_n(x,t) = a_{n}(t)p_{n+1}(x,t) + a_{n-1}(t) p_{n-1}(x,t), \quad n \geqslant 0, 
$$ 
with $p_{-1}(x,t)=0$ and $p_{0}(x,t)=\mu_{0}(t)^{-1/2}$.

It is well known that the coefficients $ a_{n}(t)$ satisfy the difference equation 
\begin{equation} \label{dif_eq_1_var}
4\,a_{n}^2(t) [a_{n+1}^2(t) + a_{n}^2(t) + a_{n-1}^2(t)] -  2\, t \,a_{n}^2(t) = n+1   , \quad n \geqslant 0,
\end{equation}
where $a_0^{2}(t) =\mu_2(t)/\mu_0(t)$ and $a_{-1}(t) =0$ (see, for instance, \cite{BR94, Ma86, Mag99, Va08}). 

Also well known is the fact that the difference equation \eqref{dif_eq_1_var} coincides with the discrete Painlev\'{e} equation dPI
\begin{equation*} 
x_n (x_{n+1} + x_n + x_{n-1}) -  \delta\,x_n = \alpha\, n +\beta +(-1)^n \,\gamma, 
\end{equation*}
with $x_n = a^2_{n}(t), \alpha =  \beta = 1/4, \gamma =0, \delta = t/2.$ See more about relations between orthogonal polynomials and Painlev\'{e} equations in \cite{Va08} and the references therein.

If we consider the sequence of monic orthogonal polynomials associated with $w_{t}(x)$, $\{q_n(x,t)\}_{n\geqslant 0}$, again a sequence of polynomials in the variable $x$ and whose coefficients depend on $t$, it satisfies 
$$
 x \,q_n(x,t) =  q_{n+1}(x,t) + \beta_{n}(t) q_{n-1}(x,t), \quad n \geqslant 0, 
$$
with $q_{-1}(x,t)=0$,  $q_{0}(x,t)=1$ and $\beta_{n}(t) = a_{n-1}^{2}(t)$. The coefficients $\beta_{n}(t)$ satisfy the Langmuir lattice (or Volterra lattice)
\begin{equation} \label{Lang_one_v}
  \dot{\beta}_{n}(t) = \beta_{n}(t) [\beta_{n+1}(t)-\beta_{n-1}(t)], \quad 
 n \geqslant 0,
\end{equation}
where, as usual, $ \dot{\beta}_{n}(t) = \dfrac{d}{dt} \beta_{n}(t)$, see \cite{Pe01}.

Consequently, the Langmuir lattice in terms of $a_{n}(t)$ is
$$
 \dot{a}_{n}(t) = \frac{a_{n}(t)}{2} [a_{n+1}^2(t)-a_{n-1}^2(t)], \quad n \geqslant 1. 
$$

\medskip
 
The connection between the coefficients of the three term recurrence relation for orthogonal polynomials in one variable and Painlev\'{e} equations  (\cite{Va08}), Langmuir or Toda lattices (\cite{Pe01}) is very well known. 
A fundamental paper regarding discrete Painlev\'e I and Laguerre-Freud equations is \cite{Mag99}. 
The motivation of this manuscript is to analyse extensions of  the equation dPI, showing that
the matrix coefficients of three term relations of two variable Freud orthogonal polynomials satisfy some matrix difference equations, that we call matrix Painlev\'e-type difference equations, and also to present two dimensional version of the Langmuir lattices. There are previous papers dealing with the extension of such systems to the matrix realm.  In \cite{Cas12,Gru11} the matrix extension of dPI was first derived using the Riemann-Hilbert problem for the theory of matrix orthogonal polynomials. This has been extended further to alt-dPI, dPII and dPIV, see \cite{Bra21,Bra21_2,Bra22,Cas19}. Matrix Painlev\'e systems have been also studied in \cite{Caf14,Caf18}. Confinement of singularities is a very interesting property for non-linear discrete system derived within orthogonal polynomial
theory (\cite{Mas19,Ram91}), for its application for matrix dPI see \cite{Cas14}.

 As it is well known, the study of bivariate orthogonal polynomials is not developed as deeply as in the univariate case. The first difficulty lies in the fact that there is no unique orthogonal system, due to the fact that several orderings of the bivariate monomials are possible. Therefore, it is necessary to fix an order on the monomials, to choose a representation for the polynomials and develop the theory. In this paper, we use the vector representation for polynomials in two variables introduced in \cite{Ko1}, \cite{Ko2}, and developed in \cite{Xu93}. There, the graded lexicographical order is used, and the representation of the polynomials as vectors whose entries are independent polynomials of the same total degree is introduced. 
However, the size of these vectors and the corresponding coefficient matrices of the formulas are increasing with the degree,  on the contrary to the non-matrix case, where the size is constant.  

In \cite{Sue99}, the vector representation for general families of bivariate orthogonal polynomials is not used, but main properties as three term relations for the orthogonal polynomials appear in a non-matrix formulation. Despite to the fact that the vector-matrix representation apparently adding more complexity to the problem, the vector representation of the families of orthogonal polynomials and the vector-matrix formulation of the three term relations, that first appeared in \cite{Ko2}, has proven to be a very powerful tool when formulating results in the bivariate environment, simplifying the notations. Now, the involved coefficients are, in general, rectangular matrices of increasing size. Nevertheless, the vector-matrix notation must be interpreted as a compact form to express properties that could be write in another form, as, for instance, in \cite{Sue99}. 

The aim of this paper is to investigate the symmetric bivariate Freud weight function given by
$$
W(x,y) =e^{-q(x,y)},\qquad (x,y) \in \mathbb{R}^2,
$$
where 
$$
q(x,y) = a_{4,0} \, x^4 + a_{2,2} \, x^2 \, y^2 + a_{0,4}\, y^4 + a_{2,0} \, x^2 + a_{0,2} \, y^2
$$
and $a_{i,j}$ are real parameters.  We analyse the bivariate orthonormal polynomials with respect to $W(x,y)$ by using, as the main tool, the vector representation for the families of orthogonal polynomials. In this environment, we can formulate the main properties in a vector-matrix form, deducing and writing the properties in a friendly form extending the results in one variable to the bivariate case.

We extend the difference equation \eqref{dif_eq_1_var} for the matrix coefficients of the three term relations for these polynomials when $a_{2,0} = a_{0,2} = - t$, getting matrix Painlev\'e-type difference equations for the respective coefficients.  We also present 2D Langmuir lattices for the matrix coefficients of the three term relations satisfied by the orthogonal polynomial systems associated with $W(x,y)$, where $a_{2,0} = a_{0,2} = - t$, $t \in \mathbb{R}$. Furthermore, matrix differential-difference equations are provided for the orthogonal polynomial systems.

This paper is structured as follows. In Section \ref{sec_basic_tools} we recall the basic results about bivariate polynomials in vector-matrix representation that we need along the paper. 

In Section \ref{sec_bi_Fr_we_fun} we present the Freud inner product associated with the bivariate Freud weight function that is considered in this work. 
The three term relations satisfied by the bivariate orthonormal polynomials and the involved matrix coefficients are  given. The structure relations as well as a differential-difference equation satisfied by the orthonormal polynomials system are also presented. These structure relations are related to the matrix Pearson-type equation satisfied by the bivariate Freud weight function. 

In Section \ref{sec_some_results}, relations for the coefficients of the three term relations and  for the coefficients of the structure relations for orthonormal polynomials are presented. We also give non-linear four term relations for the coefficients of the three term relations for orthonormal polynomials.

In Section \ref{sec_Freud_type_} we present the main results, that are the matrix Painlev\'e-type difference equations for the coefficients of the three term relations of the orthonormal polynomial system. They are 
extensions for two variables for the difference equation \eqref{dif_eq_1_var}, see Theorem \ref{Theo_TTR_A}. 

Furthermore, considering the Freud weight function
$W(x,y) = e^{-q(x,y)}$, with 
$q(x,y)={a_{4,0}x^4 + a_{2,2}x^2y^2 + a_{0,4}y^4 - t (x^2 + y^2)},$
depending on the real parameter $t$, 2D Langmuir lattices for the coefficients of the three term relations are given in Section \ref{sec_Langmuir_lattice}.

\medskip

\section{Basic tools} 
\label{sec_basic_tools}

We start introducing the basic definitions and main tools that we will need along the paper. We refer mainly \cite{DX14}.

Let us consider the linear space of real polynomials in two variables $x$ and $y$
$$
\Pi = \mathrm{span} \langle x^h\,y^k: h, k \geqslant 0\rangle,
$$
and we define the linear space
$$
\Pi_n = \mathrm{span} \langle x^h\,y^k: h+ k \leqslant n\rangle,
$$
of finite dimension $(n+1)(n+2)/2$. Observe that $\cup_{n\geqslant 0} \Pi_n = \Pi.$

As usual, a two variable polynomial of (total) degree $n$, i.e., $p(x,y)\in \Pi_n$, is given by
$$
p(x,y) = \sum_{h+k\leqslant n} c_{h,k}\, x^h \, y^k, \quad c_{h,k}\in \mathbb{R}.
$$
Now we define the vector representation for bivariate polynomials introduced in \cite{Ko1}, \cite{Ko2}, and developed in \cite{Xu93}, by using the graded lexicographical order. Notice that the size of the vectors is increasing with the degree.

\begin{definition}
A \emph{polynomial system (PS)} 
is a sequence of vectors of polynomials $\{\mathbb{P}_n\}_{n\geqslant0}$ of increasing size $(n+1)$  
$$
\mathbb{P}_n = (P_{n,0}(x,y), P_{n,1}(x,y), \ldots, P_{n,n}(x,y))^T,
$$
such that every bivariate polynomial $P_{n,i}(x,y)$ has exactly degree $n$ and the set $\{P_{n,0}(x,y)$, $P_{n,1}(x,y)$, $\ldots$, $P_{n,n}(x,y)\}$ is linearly independent.
\end{definition}
Observe that $\{\mathbb{P}_m\}_{m=0}^n$ contains a basis of $\Pi_n$, and, by extension, we will say that $\{\mathbb{P}_m\}_{m=0}^n$ is a basis of $\Pi_n$.

The simplest PS is the so-called {\it canonical basis} $\{\mathbb{X}_n\}_{n\geqslant0}$, defined as 
$$
\mathbb{X}_n = (x^n, x^{n-1}\,y, x^{n-2}\,y^2, \ldots, x\,y^{n-1}, y^{n})^T.
$$
Following \cite{DX14}, observe that
\begin{equation} \label{xXLX}
x\,\mathbb{X}_n = x\,\begin{pmatrix}
x^n\\
x^{n-1}\, y\\
x^{n-2}\, y^2\\
\vdots\\
x\, y^{n-1}\\
y^n
\end{pmatrix} = \begin{pmatrix}
x^{n+1}\\
x^{n}\, y\\
x^{n-1}\, y^2\\
\vdots\\
x^2\, y^{n-1}\\
x\,y^n
\end{pmatrix} = L_{n,1}\,\mathbb{X}_{n+1},
\end{equation}
for $n\geqslant 0$, analogously, $y\,\mathbb{X}_n = L_{n,2}\,\mathbb{X}_{n+1}$, where $L_{n,1}$ and $L_{n,2}$ are $(n+1)\times (n+2)$ matrices given by
\begin{equation} \label{L1L2}
L_{n,1} = \left(\begin{array}{ccc|c}
1        &        & \bigcirc & 0     \\
         & \ddots &          & \vdots \\
\bigcirc &        & 1        & 0 
\end{array}\right)
\quad \mbox{and} \quad 
L_{n,2} = \left(\begin{array}{c|ccc}
0      & 1        &        & \bigcirc \\
\vdots &          & \ddots &          \\
0      & \bigcirc &        & 1         
\end{array}\right),
\end{equation}
where the symbol $\bigcirc$ represents a triangle of zero elements of adequate size. This notation will be used along this work.
Observe that $L_{n,i}$ are full rank matrices,
such that $L_{n,i}\,L_{n,i}^T = I_{n+1}$.

We can write 
\begin{equation}\label{partial_x}
\partial_x\,\mathbb{X}_n = \partial_x\,\begin{pmatrix}
x^n\\
x^{n-1}\, y\\
x^{n-2}\, y^2\\
\vdots\\
x\, y^{n-1}\\
y^n
\end{pmatrix} = \begin{pmatrix}
n\,x^{n-1}\\
(n-1)\,x^{n-2}\, y\\
(n-2)\,x^{n-3}\, y^2\\
\vdots\\
y^{n-1}\\
0
\end{pmatrix} = L_{n-1,1}^T\,N_{n,1}\,\mathbb{X}_{n-1},
\end{equation}
moreover, $\partial_y\,\mathbb{X}_n = L_{n-1,2}^T\,N_{n,2}\,\mathbb{X}_{n-1}$, where 
\begin{equation}\label{N_n}
N_{n,1} = \begin{pmatrix}
n        &     &        & \bigcirc \\
         & n-1 &        &          \\
         &     & \ddots &         \\
\bigcirc &     &        & 1       
\end{pmatrix} 
\quad \mbox{and} \quad 
N_{n,2} = \begin{pmatrix}
1        &   &        & \bigcirc \\
         & 2 &        &          \\
         &   & \ddots &         \\
\bigcirc &   &        & n       
\end{pmatrix}.
\end{equation}

Let $\{\mathbb{P}_n\}_{n\geqslant0}$ be a PS. There exist matrices of constants $G^n_k$ of respective sizes $(n+1)\times (k+1)$ such that every vector polynomial $\mathbb{P}_n$ can be express in terms of the canonical basis  
\begin{equation*} 
\mathbb{P}_n = G_n \,\mathbb{X}_n + G_{n-1}^{n} \,\mathbb{X}_{n-1} +
 G_{n-2}^{n}\, \mathbb{X}_{n-2} + \cdots  + G_{1}^{n}\, \mathbb{X}_{1}  + G_{0}^{n} \, \mathbb{X}_{0},
\end{equation*}
where  $G_n$ is a $(n+1)$ non-singular matrix, because the independence of the entries of $\mathbb{P}_n$ and $\mathbb{X}_n$. We use the convention $G_{m}^n = \mathtt{0}$, for $m>n$ and $m<0$.

\medskip

\section{Bivariate Freud weight functions}
\label{sec_bi_Fr_we_fun}

We work with a bivariate Freud weight function in the form 
\begin{equation} \label{wf}
W(x,y) =e^{-q(x,y)},\quad (x,y) \in \mathbb{R}^2,
\end{equation}
where 
\begin{equation}\label{q(x,y)}
q(x,y) = a_{4,0} \, x^4 + a_{2,2} \, x^2 \, y^2 + a_{0,4}\, y^4 + a_{2,0} \, x^2 + a_{0,2} \, y^2,
\end{equation}
is a bivariate polynomial of degree 4, such that the coefficients $a_{4,0}, a_{2,2}, a_{0,4} \geqslant 0$, and $a_{2,0}, a_{0,2} \in \mathbb{R}$, with
$a_{4,0} + a_{2,2} > 0$ and  $a_{2,2} + a_{0,4} > 0$. 

Observe that $q(- x,- y) = q(x,y)$, for $(x,y) \in \mathbb{R}^2$, that is, $q(x,y)$ is an even function, and, as consequence, 
$W(-x,-y) = W(x,y)$. Following \cite[p.~76]{DX14}, the bivariate Freud weight function $W(x,y)$ is centrally symmetric. 

We define the bivariate Freud moment functional as
\begin{equation*} 
\langle \mathbf{u}, f \rangle = \iint_{-\infty}^{+\infty} f(x,y)\, W(x,y)\, dx\,dy,
\end{equation*}
and its associated moments as
$$
\mu_{n,m} =\langle \mathbf{u}, x^n\,y^m \rangle = \iint_{-\infty}^{+\infty} x^n\,y^m\, W(x,y)\, dx\,dy < +\infty,
$$
for $n, m = 0, 1, 2, \ldots$. Since $\mathbf{u}$ is centrally symmetric, then, for $n+m$ odd, we get
$$
\mu_{n,m} =\langle \mathbf{u}, x^n\,y^m \rangle= 0.
$$
Furthermore, since the special shape of the weight function, the moments such that $n$ or $m$ is odd are zero, that is,
$$
\mu_{n,m} =0, \quad \mathrm{for} \ n \ \mathrm{or} \ m \ \mathrm{odd}.
$$

Thus, we will consider the inner product
\begin{equation}\label{ip}
( f, g) := \langle \mathbf{u}, f\,g\rangle = \iint_{-\infty}^{+\infty} f(x,y)\,g(x,y)\,W(x,y)\, dx\,dy.
\end{equation}

\subsection{Orthonormal Polynomial Systems}

Let $\{\mathbb{P}_n\}_{n\geqslant0}$ be a polynomial system satisfying
\begin{align*}
(\mathbb{P}_n, \mathbb{P}_n^T) = \langle \mathbf{u},\mathbb{P}_n\,\mathbb{P}_n^T  \rangle &= I_{n+1},\\
(\mathbb{P}_n, \mathbb{P}_m^T) = \langle \mathbf{u},\mathbb{P}_n\,\mathbb{P}_m^T  \rangle &= \mathtt{0},
\end{align*}
where $\mathtt{0}$ is the zero matrix of adequate size. We say that $\{\mathbb{P}_n\}_{n\geqslant0}$ is an {\it orthonormal polynomial system} associated with the Freud inner product \eqref{ip}. 

Since the inner product \eqref{ip} is centrally symmetric, every vector of polynomials $\mathbb{P}_n$ reduces to
\begin{equation}\label{expl_expr}
\mathbb{P}_n = G_n \,\mathbb{X}_n + G_{n-2}^{n}\, \mathbb{X}_{n-2} + G_{n-4}^{n}\, \mathbb{X}_{n-4} + \cdots,
\end{equation}
that is, $\mathbb{P}_n$ contains only even powers if $n$ is even, or odd powers in the case of $n$ odd. The 
matrices $G^n_k$ are of order $(n+1)\times (k+1)$ and  $G_n$ is a matrix of order $(n+1)\times (n+1)$.

\subsection{Three term relations}  

Since $W(x,y)$ is an even function, the three term relations for the orthonormal polynomial system  $\{\mathbb{P}_n\}_{n\geqslant 0}$ takes the form (\cite[p. 77]{DX14}),
\begin{equation}\label{TTR-O}
\begin{aligned}
x \,\mathbb{P}_n &= A_{n,1}\,\mathbb{P}_{n+1} + A_{n-1,1}^{T}\,\mathbb{P}_{n-1},  \\
y \,\mathbb{P}_n &= A_{n,2}\,\mathbb{P}_{n+1} + A_{n-1,2}^{T}\,\mathbb{P}_{n-1}, 
\end{aligned}
\end{equation}
for $n \geqslant 0$, where $\mathbb{P}_{-1} =0$, $\mathbb{P}_0 = \mu_{0,0}^{-1/2}$, and  $A_{n,i}$, for $i=1,2$, are full rank $(n+1)\times(n+2)$ matrices. 
Observe that the $2(n+1)\times (n+2)$ joint matrix
\begin{align} \label{joint_An}
A_n = \left(\begin{array}{c}
      A_{n,1}\\
      A_{n,2}
      \end{array}\right)
\end{align}
is also a full rank matrix.

Computing directly, we get the initial terms
$$
A_{0,1} = \left(\sqrt{\frac{\mu_{2,0}}{\mu_{0,0}}}, 0\right), \quad 
A_{0,2} = \left(0,\sqrt{\frac{\mu_{0,2}}{\mu_{0,0}}}\right),
$$
since $\mathbb{P}_1 = (\mu_{2,0}^{-1/2} x, \mu_{0,2}^{-1/2} y)^T$. In this way, the leading coefficient matrices of $\mathbb{P}_0$ and $\mathbb{P}_1$ are respectively given by
$$
G_0 = \mu_{0,0}^{-1/2}, \qquad G_1 = \begin{pmatrix}
\mu_{2,0}^{-1/2} & 0 \\
0 & \mu_{0,2}^{-1/2}
\end{pmatrix}.
$$

\subsection{Pearson matrix equation for the Freud weight function}   

A direct computation on $W(x,y)$, given by \eqref{wf} and \eqref{q(x,y)},  shows that
\begin{equation}\label{Pearson} 
\begin{aligned}
\partial_x  W(x,y) &= -(4 \, a_{4,0} \, x^3 + 2 \, a_{2,2}\, x \, y^2 + 2 \, a_{2,0}\, x) \, W(x,y),\\[1ex]
\partial_y  W(x,y) &= -(2 \, a_{2,2} \, x^2 \, y + 4 \,a_{0,4} \, y^3 + 2 \, a_{0,2}\, y) \, W(x,y).
\end{aligned} 
\end{equation}

Given $M_1, M_2$, matrices of polynomials of the same order,
the divergence operator for the join matrix is defined by
\begin{equation*}
	\mathrm{div}\left(\begin{array}{c}
		M_1\\
		M_2
	\end{array}\right) = \partial_x M_1 + \partial_y M_2,
\end{equation*}
hence, we can state that
the weight function \eqref{wf} satisfies the bivariate Pearson equation
$$
\mathrm{div}(\Phi \, W(x,y) ) = \Psi^{T} \, W(x,y), 
$$
where
\[
\Phi = \begin{pmatrix}
1 & 0 \\
0 & 1 
\end{pmatrix},  \quad  \quad
\Psi = \begin{pmatrix} \psi_1\\ \psi_2 \end{pmatrix},
\]  
\begin{align*}
\psi_1 = \psi_1(x,y) &= -(4 \, a_{4,0} \, x^3 + 2 \, a_{2,2}\, x \, y^2 + 2 \, a_{2,0}\, x), \\
\psi_2 = \psi_1(x,y) &= -(2 \, a_{2,2} \, x^2 \, y + 4 \,a_{0,4} \, y^3 + 2 \, a_{0,2}\, y).
\end{align*}
Observe that $\deg \psi_1 = \deg\psi_2 = 3$.

\subsection{Structure relation and difference-differential equation}

Now using the fact that the weight function  \eqref{wf} is centrally symmetric, and the Pearson equations  \eqref{Pearson} for the weight function, we know that the orthonormal polynomial system, $\{\mathbb{P}_n\}_{n\geqslant0}$, (see \cite{AFPP07}), satisfies the following structure relations
\begin{equation}\label{rel_Est}
\begin{aligned} 
\partial_x \, \mathbb{P}_n &= B_{n,1} \,\mathbb{P}_{n-1} + C_{n,1} \,\mathbb{P}_{n-3},  \\[1ex]
\partial_y \, \mathbb{P}_n &= B_{n,2} \,\mathbb{P}_{n-1} + C_{n,2} \,\mathbb{P}_{n-3}, 
\end{aligned}
\end{equation}
 for $n\geqslant1$, where $\mathbb{P}_{-2} = \mathbb{P}_{-1} =0$, $B_{n,i}$, $C_{n,i}$ are matrices of respective sizes $(n+1) \times n$ and $(n+1) \times (n-2)$, and $C_{1,i} = C_{2,i}=0$, for $i=1,2$.

\medskip 

Following \cite{AFPP08}, since the Freud weight function \eqref{wf} is semiclassical, there exists a second order partial differential functional 
$$
\mathcal{F} = \partial_{xx} + \partial_{yy} + \psi_1\, \partial_x + \psi_2\,\partial_y
$$
such that 
\begin{equation}\label{dde}
\mathcal{F}\,\mathbb{P}_n
= \Lambda^n_{n+2}\,\mathbb{P}_{n+2} + \Lambda^n_{n}\,\mathbb{P}_{n} + \Lambda^n_{n-2}\,\mathbb{P}_{n-2},
\end{equation}
for $n\geqslant1$, where  
\begin{align*}
\Lambda^n_{n+2} & = - [B_{n,1}	\, C_{n+2,1}^T + B_{n,2}	\, C_{n+2,2}^T ], \\
\Lambda^n_{n} & = - [ B_{n,1}B_{n,1}^T + C_{n,1}C_{n,1}^T + B_{n,2}B_{n,2}^T + C_{n,2}C_{n,2}^T ], \\
\Lambda^n_{n-2} & =- [ C_{n,1}B_{n-2,1}^T + C_{n,2} B_{n-2,2}^T ] = (\Lambda^{n-2}_{n})^T,
\end{align*}
that is, the orthonormal polynomial system $\{\mathbb{P}_n\}_{n\geqslant0}$ satisfies the matrix partial-differential-difference equation \eqref{dde}.

\medskip

\section{Results involving the matrix coefficients} 
\label{sec_some_results}

In this section we show several relations between the matrix coefficients of the three term relations  for orthonormal polynomials \eqref{TTR-O}, the matrix coefficients of the structure relations \eqref{rel_Est} and the matrices involved in the explicit expressions of the vector polynomials \eqref{expl_expr}. 

We start by defining two useful matrices and establishing their relations with the Pearson-type equation for the weight function \eqref{Pearson}. 

Let us define  $(n+1) \times (n+1)$  upper and lower triangular matrices,  that involve the 
coefficients of the weight function \eqref{wf},  

\begin{equation} \label{K1}
\begin{aligned}
K_{n,1} = &
\begin{pmatrix}
4 a_{4,0} & 0         & 2 a_{2,2} &         & \bigcirc \\
          & 4 a_{4,0} & 0         &  \ddots &          \\
          &           & \ddots    & \ddots  & 2 a_{2,2}   \\
          &           &           & \ddots & 0    \\
\bigcirc  &           &           &        & 4 a_{4,0}
\end{pmatrix}
\end{aligned}
\end{equation}
and
\begin{equation} \label{K2}
\begin{aligned}
K_{n,2} = 
\begin{pmatrix}
4 a_{0,4} &           &           &          & \bigcirc \\
   0      & 4 a_{0,4} &           &          &          \\
2 a_{2,2} &  0        & \ddots    &          &           \\
          & \ddots    & \ddots    & \ddots   &    \\
\bigcirc  &           & 2 a_{2,2} &   0   & 4 a_{4,0}
\end{pmatrix}.
\end{aligned}
\end{equation}
Then, one can easily see that the  matrices $K_{n,i}$ and $L_{n,i}$ defined in \eqref{L1L2}, for $i=1,2$, are related as
\begin{equation}\label{LLKKLL} 
\begin{aligned}
L_{n,1}\,L_{n+1,1} \, K_{n+2,1} = & \, 4 \, a_{4,0}\,L_{n,1}\,L_{n+1,1} + 2 \, a_{2,2}\,L_{n,2}\,L_{n+1,2}, \\[1ex]
L_{n,2}\,L_{n+1,2} \, K_{n+2,2} = & \, 4 \, a_{0,4}\,L_{n,2}\,L_{n+1,2} + 2 \, a_{2,2}\,L_{n,1}\,L_{n+1,1}.
\end{aligned}
\end{equation}

Using the relations \eqref{LLKKLL} and the Pearson matrix equation \eqref{Pearson}, we can prove the following result.

\begin{proposition}   \label{propo_psiXLKX}  
The following hold  
\begin{equation} \label{psiXLKX}  
\begin{aligned}
\psi_1(x,y)\,\mathbb{X}_{n-1} &= - L_{n-1,1} \,L_{n,1} \,L_{n+1,1}\, K_{n+2,1}\, \mathbb{X}_{n+2} -2a_{2,0}L_{n-1,1}\mathbb{X}_{n},
\\
\psi_2(x,y)\,\mathbb{X}_{n-1} &=  - L_{n-1,2} \,L_{n,2} \,L_{n+1,2}\, K_{n+2,2}\, \mathbb{X}_{n+2}  -2a_{0,2}L_{n-1,2}\mathbb{X}_{n}.
\end{aligned}
\end{equation}
\end{proposition}

\medskip

\subsection{Explicit expressions}

Next result brings explicit expressions for the matrix coefficients  $A_{n,i}$, $i=1,2$, of the three term relations \eqref{TTR-O}, and for the matrix coefficients $B_{n,i}$ and $C_{n,i}$, $i=1,2$, defined on the structure relations \eqref{rel_Est}, in terms of the matrices $G_n$, $L_{n,i}$, $N_{n,i}$,  and $K_{n,i}$, for $i=1,2$,  defined by \eqref{expl_expr}, \eqref{L1L2}, \eqref{N_n}, and \eqref{K1}--\eqref{K2}, respectively.

\begin{proposition} 
For the matrix coefficients $A_{n,i}$,  $B_{n,i}$ and  $C_{n,i}$, $i=1,2$, of the three term relations \eqref{TTR-O} and the structure relations \eqref{rel_Est}, respectively, the following properties hold
\begin{itemize}
\item[i)]
\begin{equation}\label{A_ni} 
A_{n,i} = G_n \, L_{n,i} \,  G_{n+1}^{-1},  \quad n \geqslant 0. 
\end{equation}

\item[ii)]
\begin{equation}\label{B_ni}
B_{n,i} = G_n\,L_{n-1,i}^T\, N_{n,i} \,G_{n-1}^{-1}, \quad n \geqslant 1.  
\end{equation}

\item[iii)] 
\begin{equation}\label{C_ni}
C_{n,i}^T = G_{n-3}\,L_{n-3,i}\,L_{n-2,i}\,L_{n-1,i}\,K_{n,i}\,G_n^{-1},  \quad n \geqslant 3,
\end{equation}
\end{itemize}
where the matrices $G_n$, $L_{n,i}$, $N_{n,i}$,  and $K_{n,i}$, for $i=1,2$,  are defined by \eqref{expl_expr}, \eqref{L1L2}, \eqref{N_n},   and  \eqref{K1}-\eqref{K2}, respectively.

Moreover, the right pseudo inverse matrix of $A_{n,i}$ is
\begin{equation}\label{A_inv}
A_{n,i}^{-1} = G_{n+1} \, L_{n,i}^T \,  G_{n}^{-1},  \quad i=1,2.
\end{equation}
\end{proposition}

\begin{proof}

\noindent
i) \ Substituting the explicit expression of $\mathbb{P}_n$ \eqref{expl_expr} on the three term relation \eqref{TTR-O} 
we have
\begin{align}  \label{eq_expl}
x \Big[ G_n \, \mathbb{X}_n + G_{n-2}^{n} \mathbb{X}_{n-2} + \cdots   \Big] = & \, A_{n,1} \Big[ G_{n+1} \mathbb{X}_{n+1} + G_{n-1}^{n+1} \mathbb{X}_{n-1} + \cdots   \Big] \\
&+  A_{n-1,1}^{T} \Big[ G_{n-1} \, \mathbb{X}_{n-1} + 
 G_{n-3}^{n-1} \, \mathbb{X}_{n-3} + \cdots \Big] \nonumber
\end{align}
and  analogue to the three term relation for the second variable. Using \eqref{xXLX}, and adjusting leading coefficients, we have
$$
G_n \, L_{n,i}  = A_{n,i}\,  G_{n+1}, \quad i=1,2,
$$
and \eqref{A_ni} holds.

The pseudo inverse for $A_{n,i}$ by the right side \eqref{A_inv} follows immediately.

\medskip

\noindent
ii) \ In the same way, substituting \eqref{expl_expr} on \eqref{rel_Est} for $i=1$, we get
\begin{align}
\partial_x \Big[ G_n \, \mathbb{X}_n + G_{n-2}^{n} \mathbb{X}_{n-2} + \cdots   \Big] \label{expl_par}
=& \, B_{n,1} \Big[ G_{n-1} \mathbb{X}_{n-1} + G_{n-3}^{n-1} \mathbb{X}_{n-3} + \cdots   \Big]  \\
& +  C_{n,1} \Big[ G_{n-3} \, \mathbb{X}_{n-3} + G_{n-5}^{n-3} \mathbb{X}_{n-5} + \cdots \Big]. \nonumber
\end{align}
Next, applying \eqref{partial_x}, and adjusting leading coefficients we obtain
\begin{equation*}
G_n\,L_{n-1,1}^T\, N_{n,1} = B_{n,1}\,G_{n-1}. 
\end{equation*}
Doing analogue for $i=2$ we get \eqref{B_ni}.

\medskip

\noindent
iii) \  Multiplying the structure relation \eqref{rel_Est} for $i=1$ by $\mathbb{P}_{n-3}^T$, and applying the inner product \eqref{ip}, we get
\begin{align*}
\langle \mathbf{u}, \partial_x [\mathbb{P}_{n}] \mathbb{P}_{n-3}^T \rangle = & \,  
B_{n,1} \langle \mathbf{u}, \mathbb{P}_{n-1}\, \mathbb{P}_{n-3}^T \rangle
+ C_{n,1} \langle \mathbf{u}, \mathbb{P}_{n-3}\, \mathbb{P}_{n-3}^T \rangle,
\end{align*}
that is,
$
C_{n,1} = \langle \mathbf{u}, \partial_x [\mathbb{P}_{n}] \mathbb{P}_{n-3}^T \rangle.
$
Then,
$$
C_{n,1}  = \langle \mathbf{u}, \partial_x [\mathbb{P}_{n} \mathbb{P}_{n-3}^T ] \rangle -
 \langle \mathbf{u},  \mathbb{P}_{n} \partial_x[ \mathbb{P}_{n-3}^T] \rangle
 =  \langle \mathbf{u}, \partial_x [\mathbb{P}_{n} \mathbb{P}_{n-3}^T ] \rangle,
$$
because the orthogonality. Integrating $C_{n,1} $ by parts on the variable $x$, taking into account the  behaviour of the weight function on $\mathbb{R}^2$, i.e., for $F(x,y) \in \Pi$, the value of $F(x,y)W(x,y)$ goes to zero when the variables $x$ and $y$ diverges positive or negatively,
and using the first Pearson equation for the weight function \eqref{Pearson}, we deduce
\begin{align*}
C_{n,1}  = & \, \iint_{-\infty}^{+\infty} \partial_x [ \mathbb{P}_n \, \mathbb{P}_{n-3}^T] \, W(x,y)\, dx\, dy  
\, = \, - \iint_{-\infty}^{+\infty} \mathbb{P}_n \,\mathbb{P}_{n-3}^T \, \partial_x \,  W(x,y)\, dx\, dy \\
= & \, - \iint_{-\infty}^{+\infty}  \mathbb{P}_n \,\mathbb{P}_{n-3}^T \,\psi_1(x,y)\,  W(x,y)\, dx\, dy. 
\end{align*}
Using the explicit expression \eqref{expl_expr} of $\mathbb{P}_{n-3}$ and relations \eqref{psiXLKX} of Proposition \ref{propo_psiXLKX}, we deduce that $\mathbb{P}_{n-3} \psi_1(x,y)$ is a $(n-2)\times 1$ vector polynomial of degree $n$. Hence,
$$
C_{n,1} = \iint_{-\infty}^{+\infty} \mathbb{P}_n \mathbb{P}_n^T \, W(x,y)\, dx\, dy \ (G_n^{-1})^{T}\, K_{n,1}^{T}\,L_{n-1,1}^T\,L_{n-2,1}^T\,L_{n-3,1}^{T} \,G_{n-3}^T,
$$
and \eqref{C_ni} holds for $i=1$. Similar calculation can be done for $i=2$.
\end{proof}

\medskip

Next result gives relations involving the matrix coefficients $A_{n,i}$,  $B_{n,i}$ $C_{n,i}$, for $i=1,2$, by themselves.

\begin{proposition}  
The   matrix coefficients $A_{n,i}$,  $B_{n,i}$ and  $C_{n,i}$, $i=1,2$, of the three term relations \eqref{TTR-O} and of the structure relations \eqref{rel_Est}, respectively, are related as follow
\begin{itemize}
\item[i)]
\begin{equation} \label{B_ni_A}
B_{n,i} = A_{n-1,i}^{-1}\, G_{n-1}\,N_{n,i} \,G_{n-1}^{-1}, \quad  n\geqslant1.  
\end{equation}

\item[ii)]
\begin{equation} \label{CnAn}
C_{n,i}^T = A_{n-3,i}\,A_{n-2,i}\,A_{n-1,i}\,G_n\,K_{n,i}\,G_n^{-1}, \quad n \geqslant 3.
\end{equation}

\item[iii)] 
\begin{equation} \label{PropCn}
C_{n,i} =  G_{n-2}^{n} G_{n-2}^{-1} B_{n-2,i} - B_{n,i} G_{n-3}^{n-1}G_{n-3}^{-1}, \quad n \geqslant 3,
\end{equation}
\end{itemize}
where the matrices $G^n_{n-k}$, $L_{n,i}$, $N_{n,i}$,  and $K_{n,i}$, for $i=1,2$,  are defined by \eqref{expl_expr}, \eqref{L1L2}, \eqref{N_n}, and  \eqref{K1}-\eqref{K2}, respectively.
 
Moreover, the following relations hold
\begin{equation} 
A_{n,i} G_{n-2k+1}^{n+1} + A_{n-1,i}^{T} G_{n-2k+1}^{n-2} = G_{n-2k}^{n} L_{n-2k,i} \label{GLA_i}
\end{equation}
and
\begin{equation} 
B_{n,i} G_{n-2k-1}^{n-1} + C_{n,i} G_{n-2k-1}^{n-3} = G_{n-2k}^{n} L_{n-2k-1,i}^T N_{n-2k,i}, \label{GLBC_i}
\end{equation}
for $\ k=0,1,\ldots, \lfloor n/2 \rfloor$. \\

\end{proposition}

\begin{proof}
\eqref{B_ni_A} is deduced using the explicit expression of $A_{n-1,i}^{-1}$ in \eqref{B_ni}, and \eqref{CnAn} using the relation \eqref{A_ni}  in \eqref{C_ni}.

The expression \eqref{GLA_i} is deduced adjusting the coefficients of $\mathbb{X}_{n-1}, \mathbb{X}_{n-3}, \ldots$ in \eqref{eq_expl}, and \eqref{GLBC_i} is obtained in the same way in \eqref{expl_par}.

Finally, using $k=1$ in \eqref{GLBC_i}, we get
$$
C_{n,i} G_{n-3} = G_{n-2}^{n} L_{n-3,i}^T N_{n-2,i} -  B_{n,i} G_{n-3}^{n-1}.
$$
Since $B_{n-2,i} G_{n-3} = G_{n-2} L_{n-3,i}^T N_{n-2,i} $, we can write 
\begin{align*}
C_{n,i} G_{n-3} &= G_{n-2}^{n} G_{n-2}^{-1}  G_{n-2} L_{n-3,i}^T N_{n-2,i} -  B_{n,i} G_{n-3}^{n-1}\\
& =   G_{n-2}^{n} G_{n-2}^{-1} B_{n-2,i} G_{n-3} -  B_{n,i} G_{n-3}^{n-1}
\end{align*}
hence, we get  \eqref{PropCn}.
\end{proof}

We remark that, for the general case, equations \eqref{GLA_i} and \eqref{GLBC_i} can be found in \cite{MMPP18}.

\medskip

\subsection{Non-linear four term relations for the coefficients of the three term relations.} 

We now show non-linear four term relations for the matrix coefficients of the three term relations $A_{n,i}, i=1,2$. We must remark that the results given in this subsection hold for every centrally symmetric weight function, since structure relations have not been used.
 
\begin{proposition}  
The matrix coefficients of the three term relations for orthonormal polynomials, $A_{n,i}$,
 $n\geqslant 0$  and $i,j=1,2$, satisfy
\begin{equation} \label{FTR_An_1}
A_{n,i} A_{n,j}^T + A_{n-1,i}^T {A}_{n-1,j} = G_{n-2}^{n} G_{n-2}^{-1} A_{n-2,i} {A}_{n-1,j} - A_{n,i} {A}_{n+1,j} G_{n}^{n+2} G_{n}^{-1},
\end{equation}
where the matrices $G^n_{n-k}$  are defined in \eqref{expl_expr}.
\end{proposition}

\begin{proof}
First we compute  $\langle \mathbf{u}, x^2 \mathbb{P}_n \mathbb{P}_n^T \rangle$ using the three term relation \eqref{TTR-O} and the orthogonality. Hence,
\begin{align*}
\langle \mathbf{u}, x^2 \mathbb{P}_n \mathbb{P}_n^T \rangle &=
\langle \mathbf{u}, [A_{n,1}\,\mathbb{P}_{n+1}  + A^T_{n-1,1}\,\mathbb{P}_{n-1}] [\mathbb{P}_{n+1}^T\,A_{n,1}^T  + \mathbb{P}_{n-1}^T\,A_{n-1,1}] \rangle \\
&= A_{n,1}\, A_{n,1}^T  +  A^T_{n-1,1}\,A_{n-1,1}. 
\end{align*}

Since the entries of the sequence of vectors $\{\mathbb{P}_{n}\}_{n\geqslant 0}$ form a basis for the space $\Pi$, then $x^2 \mathbb{P}_n $ can be written as
\begin{align} \label{xxPF}
x^2 \mathbb{P}_n &= F_{n+2,1}^{n} \mathbb{P}_{n+2} + F_{n,1}^{n} \mathbb{P}_{n} +  F_{n-2,1}^{n} \mathbb{P}_{n-2} + \cdots ,
\end{align}
where $F_{j,1}^{n}$ are real matrices of order $(n+1) \times (j+1)$.
On the one hand, using \eqref{expl_expr}, we get
\begin{align}
x^2 \mathbb{P}_n = & \   
F_{n+2,1}^{n} [G_{n+2} \,\mathbb{X}_{n+2} + G_{n}^{n+2}\, \mathbb{X}_{n} + G_{n-2}^{n+2}\, \mathbb{X}_{n-2} + \cdots] \nonumber \\
&+ F_{n,1}^{n} [G_n \,\mathbb{X}_n + G_{n-2}^{n}\, \mathbb{X}_{n-2} + G_{n-4}^{n}\, \mathbb{X}_{n-4} + \cdots] \label{eq_equal_1}\\
&+ F_{n-2,1}^{n} [G_{n-2} \,\mathbb{X}_{n-2} + G_{n-4}^{n-2}\, \mathbb{X}_{n-4} + G_{n-6}^{n-2}\, \mathbb{X}_{n-6} + \cdots] + \cdots.
\nonumber
\end{align}
On the other hand,  we can write  $x^2 \mathbb{P}_n $ as
\begin{align}
x^2 \mathbb{P}_n & =  \,   
x^2 [G_n \,\mathbb{X}_n +G_{n-2}^{n}\, \mathbb{X}_{n-2} + G_{n-4}^{n}\, \mathbb{X}_{n-4} + \cdots] \nonumber   \\
& =  \, G_n L_{n,1} L_{n+1,1} \mathbb{X}_{n+2} + G_{n-2}^{n} L_{n-2,1} L_{n-1,1} \mathbb{X}_n + \cdots.\label{eq_equal_2}
\end{align}
Adjusting the coefficients of the terms of $\mathbb{X}_{n+2}$ and $\mathbb{X}_{n}$ on  \eqref{eq_equal_1} and  \eqref{eq_equal_2}, we get
\begin{align*}
F_{n+2,1}^{n}G_{n+2}  & = G_n L_{n,1} L_{n+1,1}, \\
F_{n+2,1}^{n}G_{n}^{n+2} +  F_{n,1}^{n} G_n  &  =  G_{n-2}^{n} L_{n-2,1} L_{n-1,1}.
\end{align*}
Therefore
$
F_{n+2,1}^{n} = G_n L_{n,1} L_{n+1,1} G_{n+2}^{-1} \, ,
$ 
and
\begin{align*}
F_{n,1}^{n} &=  G_{n-2}^{n} L_{n-2,1} L_{n-1,1} G_n^{-1} - G_n L_{n,1} L_{n+1,1} G_{n+2}^{-1}G_{n}^{n+2} G_n^{-1} \\
&= G_{n-2}^{n}G_{n-2}^{-1} G_{n-2} L_{n-2,1} L_{n-1,1} G_n^{-1} - G_n L_{n,1} L_{n+1,1} G_{n+2}^{-1}G_{n}^{n+2} G_n^{-1}.
\end{align*}
From \eqref{A_ni},  $G_{n-2} L_{n-2,1} L_{n-1,1} G_n^{-1} = A_{n-2,1}A_{n-1,1}$,  we obtain
\begin{align}\label{relF_n1A_n}
F_{n,1}^{n} &=  G_{n-2}^{n} G_{n-2}^{-1} A_{n-2,1}  {A}_{n-1,1} - A_{n,1} {A}_{n+1,1}  G_n^{n+2} G_n^{-1}.
\end{align}
Finally, since
$
\langle \mathbf{u}, x^2 \mathbb{P}_n \mathbb{P}_n^T \rangle =
F_{n,1}^{n} \langle \mathbf{u}, \mathbb{P}_n \mathbb{P}_n^T \rangle =
F_{n,1}^{n},
$
then, for $ n \geqslant 0$,
$$
A_{n,1} A_{n,1}^T + A_{n-1,1}^T {A}_{n-1,1} = G_{n-2}^{n} G_{n-2}^{-1} A_{n-2,1}  {A}_{n-1,1} - A_{n,1} {A}_{n+1,1} G_n^{n+2} G_n^{-1}.
$$ 

Similar reasoning using $y^2 \mathbb{P}_n $, $x\,y \mathbb{P}_n $ and $y\,x \mathbb{P}_n $ gives the results. 
\end{proof}

Let us consider the joint matrix $A_n$, given by \eqref{joint_An}, of order $2(n+1) \times (n+2)$,  and the joint matrix of order $(n+1) \times 2(n+2)$,  denoted by $\bar{A}_n$, and defined by
\begin{equation*}
\bar{A}_n = \left(\begin{array}{cc}
A_{n,1}, & 
A_{n,2} 
\end{array} \right).
\end{equation*}

Remember that the Kronecker product of  $A=[a_{ij}]$, matrix of order $m \times n$, and $B=[b_{ij}]$, matrix of order $p \times q$, denoted by $A \otimes B$, is defined as the following block matrix 
		$$ A \otimes B = \left(\begin{matrix}
			a_{11} B & \dots & a_{1n} B \\
			\vdots & \ddots & \vdots \\
			a_{m1} B & \dots & a_{mn} B
		\end{matrix}\right),$$
of order $mp \times nq$, see also \cite[p. 243]{HJ91}.

A direct use of the definition of Kronecker product  and equations \eqref{FTR_An_1} yields the following result.

\begin{corollary}
The sequences of the joint matrices  $A_n$ and $\bar{A}_n$ satisfy
$$
A_n A_n^T + \bar{A}_{n-1}^T \bar{A}_{n-1} = (I_{2} \otimes G_{n-2}^{n} G_{n-2}^{-1}) A_{n-2}  \bar{A}_{n-1} - A_n  \bar{A}_{n+1} 
(I_{2}  \otimes G_{n}^{n+2} G_{n}^{-1}).
$$
\end{corollary}

We observe that the matrices $F_{n,i}^{n+2}$, $F_{n,i}^{n}$ and $F_{n,i}^{n-2}$, for $i=1,2$  given in \eqref{xxPF} satisfy another interesting relation. 

\begin{corollary}
Let $F_{m,i}^{n}$ be the matrix coefficients defined in \eqref{xxPF}, for $n\geqslant2$, $i=1,2$ and $0\leqslant m\leqslant n+2$. Then
\begin{align*}
F_{n,i}^{n} =  G_{n-2}^{n} G_{n-2}^{-1} (F_{n-2,i}^{n})^T - F_{n+2,i}^n  G_n^{n+2} G_n^{-1}.
\end{align*}

\end{corollary}

\begin{proof}
For simplicity here we denote the variable $x$ by $x_1$ and the variable $y$ by $x_2$, then using the three term relations  \eqref{TTR-O}, for $i=1,2$, 
$$
x_i^2\mathbb{P}_n  = A_{n,i}A_{n+1,i}\mathbb{P}_{n+2} + (A_{n,i}A_{n,i}^T + A_{n-1,i}^TA_{n-1,i})\mathbb{P}_n + A_{n-1,i}^TA_{n-2,i}^T \mathbb{P}_{n-2}.
$$
Comparing this expression with \eqref{xxPF}, we obtain
\begin{equation}\label{comparF_n}
\begin{aligned} 
	F_{n+2,i}^{n} & = A_{n,i}A_{n+1,i},  \\
	F_{n,i}^{n}   & = A_{n,i}A_{n,i}^T + A_{n-1,1}^TA_{n-1,i},  \\
	F_{n-2,i}^{n} & = A_{n-1,i}^TA_{n-2,i}^T. 
\end{aligned}
\end{equation}
Hence, from \eqref{relF_n1A_n},  and \eqref{comparF_n}, we have
\begin{align*}
F_{n,i}^{n} & =   G_{n-2}^{n} G_{n-2}^{-1} F_{n,i}^{n-2} - F_{n+2,i}^n  G_n^{n+2} G_n^{-1}, \quad n\geqslant 2.
\end{align*}
Observing  that $F_{n,i}^{n-2} = (F_{n-2,i}^{n})^T$, $n \geqslant 2$, we finally get the result.
\end{proof}

\medskip

\section{Matrix Painlev\'{e}-type difference equations}
\label{sec_Freud_type_}

In this section we obtain non-linear three term relations for the matrix coefficients, $A_{n,i}$, $i=1,2$, of the three term relations for orthonormal polynomials, \eqref{TTR-O},  extending the well known relation \eqref{dif_eq_1_var}, namely
$$
4 \,a_n^2\,(a_{n+1}^2 + a_n^2 + a_{n-1}^2) - 2 \,t\,a_n^2  = n+1,
$$
extensively studied (\cite{BR94}, \cite{Ma86}, \cite{Mag99}, \cite{Va08}, among others) to the bivariate case. We have to taking account the non-commutativity of the product of matrices. 

\medskip

We know that in bivariate case the matrix coefficients $A_{n,i}$, for $i=1,2$, of the three term relations \eqref{TTR-O}, of order  $(n+1)\times(n+2)$, take the place of the coefficients $a_n$ of the univariate case. We can now prove the following result.

\begin{theorem}[Matrix Painlev\'{e}-type difference equations] \label{Theo_TTR_A}
For $n\geqslant 0$, the following relations, for the matrix coefficients $A_{n,i}$, $i=1,2$, 
of the three term relations \eqref{TTR-O},
hold
\begin{align*}
4\,a_{4,0}\,A_{n,1}&\left[(A_{n+1,1} A_{n+1,1}^T) A^T_{n,1} +  A_{n,1}^T(A_{n,1}A_{n,1}^T + A_{n-1,1}^TA_{n-1,1})\right]\\
& + 2\,a_{2,2}\, A_{n,1} \,\left[(A_{n+1,2} A_{n+1,1}^T) A^T_{n,2} 
+ A_{n,2}^T (A_{n,1}A_{n,2}^T +  A_{n-1,1}^TA_{n-1,2})\right]\\
& + 2\,a_{2,0}\, A_{n,1}\,A^T_{n,1} = G_nN_{n+1,1}G_{n}^{-1}
\end{align*}
and
\begin{align*}
4\,a_{0,4}\,A_{n,2}&\left[(A_{n+1,2} A_{n+1,2}^T) A^T_{n,2} +  A_{n,2}^T(A_{n,2}A_{n,2}^T + A_{n-1,2}^TA_{n-1,2})\right]\\
& + 2\,a_{2,2}\, A_{n,2} \,\left[(A_{n+1,1} A_{n+1,2}^T) A^T_{n,1} 
+ A_{n,1}^T (A_{n,2}A_{n,1}^T +  A_{n-1,2}^TA_{n-1,1})\right]\\
& + 2\,a_{0,2}\, A_{n,2}\,A^T_{n,2} = G_nN_{n+1,2}G_{n}^{-1},
\end{align*}
where  $a_{4,0},a_{2,2},a_{0,4},a_{2,0},a_{0,2}$ are  the coefficients of the bivariate Freud weight function \eqref{wf}-\eqref{q(x,y)}.
\end{theorem}

\begin{proof}
By using \eqref{Pearson}, we know that  
$$
\langle \partial_x\mathbf{u}, \mathbb{P}_{n+1}\mathbb{P}_n^T\rangle = \langle \psi_1\mathbf{u}, \mathbb{P}_{n+1}\mathbb{P}_n^T\rangle.
$$
The left-hand term is given by
\begin{align*}
\langle \partial_x\mathbf{u}, \mathbb{P}_{n+1}\mathbb{P}_n^t\rangle
=& - \langle \mathbf{u}, \partial_x[\mathbb{P}_{n+1}\mathbb{P}_n^T]\rangle 
= - \langle \mathbf{u}, \partial_x[\mathbb{P}_{n+1}]\mathbb{P}_n^T\rangle - \langle \mathbf{u}, \mathbb{P}_{n+1}\partial_x[\mathbb{P}_n^T]
\rangle \\
=& - \langle \mathbf{u}, \partial_x[\mathbb{P}_{n+1}]\mathbb{P}_n^T\rangle = - B_{n+1,1},
\end{align*}
by using the structure relation \eqref{rel_Est}.

To compute the right-hand term, we apply successively the three term relations. Observe that
\begin{align*}
x^2\mathbb{P}_{n+1} =& A_{n+1,1}A_{n+2,1}\mathbb{P}_{n+3} + [A_{n+1,1}  A^T_{n+1,1} + A^T_{n,1}A_{n,1}]\mathbb{P}_{n+1}      
            +  A^T_{n,1}A^T_{n-1,1}\mathbb{P}_{n-1},\\
x^3\mathbb{P}_{n+1} =& A_{n+1,1}A_{n+2,1}A_{n+3,1}\mathbb{P}_{n+4} \\
           & + [A_{n+1,1} A_{n+2,1} A^T_{n+2,1} + A_{n+1,1}A^T_{n+1,1}A_{n+1,1} + A_{n,1}^T A_{n,1} A_{n+1,1}]\mathbb{P}_{n+2}     \\
           & + [A_{n+1,1} A_{n+1,1}^T A^T_{n,1} + A_{n,1}^TA_{n,1}A_{n,1}^T +  A_{n,1}^TA_{n-1,1}^TA_{n-1,1}] \mathbb{P}_n\\
           & +  A^T_{n,1}A^T_{n-1,1}A_{n-2,1}^T\mathbb{P}_{n-2},  
                \end{align*}

\begin{enumerate}[(i)]

\item Using $x\mathbb{P}_{n+1} = A_{n+1,1}\mathbb{P}_{n+2} + A^T_{n,1}\mathbb{P}_{n}$, we have
$$
\langle \mathbf{u}, x\mathbb{P}_{n+1}\mathbb{P}_n^T\rangle 
=  \langle \mathbf{u}, [A_{n+1,1}\mathbb{P}_{n+2} + A^T_{n,1}\mathbb{P}_{n}] \mathbb{P}_n^T\rangle =  A^T_{n,1}.
$$

\item Moreover,
$$
\langle \mathbf{u}, x^3\mathbb{P}_{n+1}\mathbb{P}_n^T\rangle = A_{n+1,1} A_{n+1,1}^T A^T_{n,1} + A_{n,1}^TA_{n,1}A_{n,1}^T +  A_{n,1}^TA_{n-1,1}^TA_{n-1,1}.
$$

\item Analogously, using $xy^2 = yxy$,
$$
\langle \mathbf{u}, xy^2\mathbb{P}_{n+1}\mathbb{P}_n^T\rangle = A_{n+1,2} A_{n+1,1}^T A^T_{n,2} 
+ A_{n,2}^TA_{n,1}A_{n,2}^T +  A_{n,2}^TA_{n-1,1}^TA_{n-1,2}.
$$
\end{enumerate}
Observe that
\begin{align*}
\langle \psi_1\mathbf{u}, \mathbb{P}_{n+1}&\mathbb{P}_n^T\rangle = \langle \mathbf{u}, \psi_1\mathbb{P}_{n+1}\mathbb{P}_n^T\rangle \\
=& -4a_{4,0}\langle \mathbf{u}, x^3\mathbb{P}_{n+1}\mathbb{P}_n^T\rangle - 2a_{2,2} \langle \mathbf{u}, xy^2\mathbb{P}_{n+1}\mathbb{P}_n^T\rangle - 2a_{2,0} \langle \mathbf{u}, x\mathbb{P}_{n+1}\mathbb{P}_n^T\rangle\\
=& -4a_{4,0}[A_{n+1,1} A_{n+1,1}^T A^T_{n,1} + A_{n,1}^TA_{n,1}A_{n,1}^T +  A_{n,1}^TA_{n-1,1}^TA_{n-1,1}]\\
& - 2a_{2,2} [A_{n+1,2} A_{n+1,1}^T A^T_{n,2} 
+ A_{n,2}^TA_{n,1}A_{n,2}^T +  A_{n,2}^TA_{n-1,1}^TA_{n-1,2}]\\
& - 2a_{2,0}A^T_{n,1}.
\end{align*}
Therefore, 
\begin{align*}
4 a_{4,0}&\left[(A_{n+1,1} A_{n+1,1}^T) A^T_{n,1} +  A_{n,1}^T(A_{n,1}A_{n,1}^T + A_{n-1,1}^TA_{n-1,1})\right]\\
& + 2 a_{2,2} \left[(A_{n+1,2} A_{n+1,1}^T) A^T_{n,2} 
+ A_{n,2}^T (A_{n,1}A_{n,2}^T +  A_{n-1,1}^TA_{n-1,2})\right]\\
& + 2 a_{2,0}A^T_{n,1} = B_{n+1,1}.
\end{align*}
Since $B_{n+1,1} = A^{-1}_{n,1}G_nN_{n+1}G_{n}^{-1}$, we multiply all the equation by $A_{n,i}$ by the left-hand side, and the result follows for $i=1$. Analogous calculation can be done for $i=2$.
\end{proof}

For $a_{4,0}=a_{0,4}=1$, $a_{2,2}=0$, and $a_{2,0}=a_{0,2}= -t$, expressions in Theorem \ref{Theo_TTR_A} read as
\begin{align*}
4\,A_{n,i}\left[(A_{n+1,i} A_{n+1,i}^T) A^T_{n,i} +\right. & \left.A_{n,i}^T(A_{n,i}A_{n,i}^T + A_{n-1,i}^TA_{n-1,i})\right]
 - 2\,t\,A_{n,i}\,A^T_{n,i} \\
&= G_nN_{n+1,i}G_{n}^{-1},
\end{align*}
for $i=1,2$. We can say that above expressions extend the well known Freud equation \eqref{dif_eq_1_var} for the univariate case, since here the matrix coefficients $A_{n,i}, i=1,2,$ take the same roles as the coefficients $a_n$, obey the same product and difference relations, and matrices $G_nN_{n+1,i}G_{n}^{-1}$ extend the independent term $n+1$.

In the univariate case, equation \eqref{dif_eq_1_var} is a non-linear recurrence that could determine, if no zeros occur, the
consecutive recursion coefficients. However, in the bivariate case, matrix Painlev\'{e}-type difference equations are not recurrence relations for the matrix coefficients $A_{n,i}$. The matrices $A_{n,i}$ are full rank matrices invertible only by the right hand side, and this fact prevent to use the relation as a recurrence relation to compute $A_{n+1,i}$. This fact is the same as happens with the three term relations \eqref{TTR-O}, they are not recurrence relations (\cite[p.~73]{DX14}). 

Even though the dimension of the matrix coefficients  $A_{n,i}$ grows linearly with respect to the index $n$, the matrix representation of the orthogonal polynomials yields interesting matrix difference equations and in the same formal model as the discrete Painlev\'{e} equation dPI. The use of the vector-matrix representation has allowed us to construct an extension of equation \eqref{dif_eq_1_var} that reads in a similar way. Theorem \ref{Theo_TTR_A} could be proved without matrix formulation as in \cite{Sue99}, but the expressions would have read in a very cumbersome way.

\medskip

\section{2D Langmuir lattices}
\label{sec_Langmuir_lattice}

The aim of this section is to deduce formal 2D Langmuir lattices associated with a Freud weight function in two variables. As in the previous sections, our results involve matrices of increasing size and can be read as extensions of the univariate Langmuir lattices.

We assume that the coefficients of the polynomial $q(x,y)$ in \eqref{q(x,y)} satisfies  $a_{2,0} = a_{0,2}=-t$, with $t \in \mathbb{R}$, then the weight function is given by
$$
W_t(x,y) = e^{-(a_{4,0}x^4 + a_{2,2}x^2y^2 + a_{0,4}y^4) + t (x^2 + y^2)}, \quad  (x,y) \in \mathbb{R}^2.
$$
We consider the inner product
\begin{equation}\label{ip_t}
( f, g)_t := \langle \mathbf{u}_t, f\,g\rangle = \iint_{-\infty}^{+\infty} f(x,y)\,g(x,y)\,W_t(x,y)\, dx\,dy
\end{equation}
that depends on a time parameter $t$. As usual, we denote the derivative of $f(t)$ with respect to $t$ by  $\dot{f} =\dfrac{d}{dt}f(t)$.

As the univariate case, to deduce Langmuir lattices we will need a bivariate monic polynomial system 
$\{\mathbb{Q}_n(x,y,t)\}_{n \geqslant 0}\equiv \{\mathbb{Q}_n(t)\}_{n \geqslant 0}$ orthogonal with respect to the inner product \eqref{ip_t} and depending on $t$. Here $\mathbb{Q}_n(t)$ is a vector of monic polynomials on the variables $(x,y)$ such that its coefficients depend on the parameter $t$.  For $n\geqslant0$, we say that $\mathbb{Q}_n(t)$ is monic if the matrix $G_n(t)$ in its explicit expression \eqref{expl_expr} is the identity matrix $I_{n+1}$. In this case,
\begin{align*}
&  (\mathbb{Q}_n(t), \mathbb{Q}_n(t)^T) = \langle \mathbf{u},\mathbb{Q}_n(t)\,\mathbb{Q}_n(t)^T  \rangle = H_{n}(t), \\
&  (\mathbb{Q}_n(t), \mathbb{Q}_m(t)^T) = \langle \mathbf{u},\mathbb{Q}_n(t)\,\mathbb{Q}_m(t)^T  \rangle = \mathtt{0}, 
\end{align*}
where $H_n = H_{n}(t)$ is a $(n + 1)$ symmetric and positive definite matrix depending on $t$ and again $\mathtt{0}$ is the zero matrix of adequate size. 

The coefficients of the three term relations for $\{\mathbb{Q}_n(t) \}_{n \geqslant 0}$ also depends on $t$. Since the inner product \eqref{ip_t} is centrally symmetric, the three term relations take the form  
\begin{equation} \label{TTRmonict}
\begin{aligned}
 x \, \mathbb{Q}_n(t) = L_{n,1} \mathbb{Q}_{n+1}(t) + E_{n,1}(t) \mathbb{Q}_{n-1}(t),  \\
 y \, \mathbb{Q}_n(t) = L_{n,2} \mathbb{Q}_{n+1}(t) + E_{n,2}(t) \mathbb{Q}_{n-1}(t), 
\end{aligned}
\end{equation}
for $ n\geqslant 0$, where $\mathbb{Q}_{-1}(t)=0$, $\mathbb{Q}_{0}(t)=1$, and for $i=1,2$, the matrices $L_{n,i}$ were defined in \eqref{L1L2} and $E_{n,i}(t)$ are matrices of order $(n+1) \times n$, (see \cite[p. 70]{DX14}).
The matrices  $E_{n,i}(t)$ also satisfy
\begin{equation} \label{EHHE}
 E_{n,i}(t) H_{n-1}(t) = H_n(t) L_{n-1,i}^{T}, \quad i=1,2. 
\end{equation} 

\medskip

Next, we find the following relation between $\dot{H}_n(t)$ and $H_n(t)$. 

\begin{lemma} \label{dHVH}
For $n \geqslant 0$,  
$$
\dot{H}_n(t) = V_{n+1}(t) H_{n}(t),
$$
where 
\begin{equation}\label{V_n}
V_{n+1}(t) = L_{n,1} E_{n+1,1}(t) + L_{n,2} E_{n+1,2}(t) + E_{n,1}(t) L_{n-1,1}  + E_{n,2}(t) L_{n-1,2}.
\end{equation}
\end{lemma}
\begin{proof} Since 
$
\dot{W_t}(x,y) =  (x^2 + y^2)W_t(x,y),
$
we can write
\begin{align*}
\dot{H}_n(t) = & 
\iint_{-\infty}^{+\infty} \dot{\mathbb{Q}}_n(t) \,\mathbb{Q}_{n}^T(t)\,    W_t(x,y)\, dx\, dy +
\iint_{-\infty}^{+\infty} \mathbb{Q}_n(t) \,\dot{\mathbb{Q}}_{n}^T(t)\,    W_t(x,y)\, dx\, dy \\ 
& 
+ \iint_{-\infty}^{+\infty} \mathbb{Q}_n(t) \,\mathbb{Q}_{n}^T(t)\,    (x^2 +y^2) W_t(x,y)\, dx\, dy.
\end{align*}
Notice that $\deg \dot{\mathbb{Q}}_n(t) < n$, hence, using the orthogonality, and the three term relations \eqref{TTRmonict}, we get the result.
\end{proof}

Now, we define the matrices
\begin{equation} \label{E=E+E}
\textbf{E}_n(t) = 
 E_{n,1}(t) + E_{n,2}(t), \quad n \geqslant 1.
\end{equation} 
We can prove that the matrices  $\textbf{E}_{n}(t)$ satisfy a two dimension version of the Langmuir lattice.

\begin{theorem}
The matrices $\textbf{E}_{n}(t)$ satisfy the 2D Langmuir lattice
\begin{equation}\label{E_dot}
\dot{\textbf{E}}_{n}(t) = V_{n+1}(t) \textbf{E}_{n}(t) -  \textbf{E}_{n}(t) V_{n}(t), \quad n \geqslant 1,
\end{equation}
where $V_{n}(t)$ is given in \eqref{V_n}.
\end{theorem}

\begin{proof}
From \eqref{EHHE} we can write
$ \dot{H}_n(t) L_{n-1,i}^{T} = \dot{E}_{n,i}(t) H_{n-1}(t) + E_{n,i}(t) \dot{H}_{n-1}(t), $ for $i=1,2$,
hence
$$
\dot{H}_n(t) [ L_{n-1,1}^{T} + L_{n-1,2}^{T} ]= [\dot{E}_{n,1}(t) +\dot{E}_{n,2}(t)] H_{n-1}(t) + [E_{n,1}(t)+E_{n,2}(t)] \dot{H}_{n-1}(t).
$$
Using Lemma \ref{dHVH} and definition \eqref{E=E+E}, we get
\begin{align*}
 V_{n+1}(t) H_{n}(t) [ L_{n-1,1}^{T} + L_{n-1,2}^{T} ] = & \ \dot{\textbf{E}}_{n}(t) H_{n-1}(t) + \textbf{E}_{n}(t)  V_{n}(t) H_{n-1}(t),
\end{align*}
hence, using \eqref{EHHE},
\begin{align*}
\dot{\textbf{E}}_{n}(t) H_{n-1}(t) =  & \ V_{n+1}(t) [ E_{n,1}(t) + E_{n,2}(t)]  H_{n-1}(t) - \textbf{E}_{n}(t) V_{n}(t) H_{n-1}(t)  .
\end{align*}
Since $H_{n-1}(t)$ is a non-singular matrix, we obtain the result.
\end{proof}

Relation \eqref{E_dot} can be seen as a formal type of 2D Langmuir lattice for the matrix coefficients of the three term relation for the monic orthogonal   polynomials. The coefficient matrices $\textbf{E}_{n}(t)$ play the same role as the coefficients $\beta_{n}(t)$ of the univariate case  \eqref{Lang_one_v}.

\medskip

Now, we return to orthonormal polynomial systems. Since $H_n(t)$ is symmetric and positive definite, there exists another symmetric and positive definite matrix $H_{n}^{1/2}(t)$, the
so-called \emph{square root} of the matrix $H_{n}(t)$ \cite[p. 440]{HJ85} such that
$
H_{n}(t) = H_{n}^{1/2}(t)\,H_{n}^{1/2}(t).
$
Let us define the polynomial system $\{\mathbb{P}_n(t)\}_{n\geqslant0}$ by means of
$$
\mathbb{P}_n(t) =  H_{n}^{-1/2}(t)\, \mathbb{Q}_n(t), \quad n\geqslant 0.
$$
Since
\begin{align*}
&  (\mathbb{P}_n(t), \mathbb{P}_n(t)^T) = (H_{n}^{-1/2}(t) \mathbb{Q}_n(t), \mathbb{Q}_n(t)^T H_{n}^{-1/2}(t))  = I_{n+1}, \\
&  (\mathbb{P}_n(t), \mathbb{P}_m(t)^T) = (H_{n}^{-1/2}(t) \mathbb{Q}_n(t), \mathbb{Q}_m(t)^T H_{m}^{-1/2}(t))=\mathtt{0}, 
\end{align*}
then $\{\mathbb{P}_n(t) \}_{n \geqslant 0}$ is an orthonormal polynomial system with respect to \eqref{ip_t}, and satisfy the three term relations \eqref{TTR-O}, where the matrices $A_{n,i} = A_{n,i}(t)$ also depend on $t$, for $n\geqslant0$. 

The matrices involved in the respective three term relations \eqref{TTR-O} and \eqref{TTRmonict} are related by 
\begin{equation*} 
A_{n,i}(t) = H_{n}^{1/2}(t) E_{n+1,i}^{T}(t)  H_{n+1}^{-1/2}(t).
\end{equation*}
Then,
\begin{equation} \label{At=HEH}
\textbf{A}_n^{T}(t) = H_{n+1}^{-1/2}(t) \textbf{E}_{n+1}(t) H_{n}^{1/2}(t), \quad n \geqslant 0,
\end{equation}
where $\textbf{A}_n(t) = A_{n,1}(t) + A_{n,2}(t)$. Deriving \eqref{At=HEH} with respect to $t$, and omitting the parameter $t$ for simplicity, we get
$$
\dot{\textbf{A}}_n^{T} = \dot{H}_{n+1}^{-1/2} \textbf{E}_{n+1} H_{n}^{1/2} +
H_{n+1}^{-1/2} \dot{\textbf{E}}_{n+1} H_{n}^{1/2} +
H_{n+1}^{-1/2} \textbf{E}_{n+1} \dot{H}_{n}^{1/2}.
$$ 
Let us analyse term by term. From \eqref{E_dot} and \eqref{At=HEH}, we obtain
\begin{align}
H_{n+1}^{-1/2} \dot{\textbf{E}}_{n+1} H_{n}^{1/2} 
& = H_{n+1}^{-1/2} [V_{n+2}\textbf{E}_{n+1} - \textbf{E}_{n+1}V_{n+1}] H_{n}^{1/2} \nonumber \\
& = H_{n+1}^{-1/2} V_{n+2} H_{n+1}^{1/2} \textbf{A}_n^{T} 
- \textbf{A}_n^{T} H_{n}^{-1/2} V_{n+1} H_{n}^{1/2}.\label{HEdotH}
\end{align}
Using the definition of $V_{n+1}$ and \eqref{At=HEH}, we have
$$
H_{n}^{-1/2} V_{n+1} H_{n}^{1/2}
= A_{n,1} A_{n,1}^{T} + A_{n,2} A_{n,2}^{T} + A_{n-1,1}^{T} A_{n-1,1}
  + A_{n-1,2}^{T} A_{n-1,2}.
$$
Substituting this relation in \eqref{HEdotH} we get 
\begin{align*} 
H_{n+1}^{-1/2} \dot{\textbf{E}}_{n+1} H_{n}^{1/2} 
 = & [A_{n+1,1} A_{n+1,1}^{T} + A_{n+1,2} A_{n+1,2}^{T} + A_{n,1}^{T} A_{n,1} + A_{n,2}^{T} A_{n,2}] \textbf{A}_n^{T} \\
& - \textbf{A}_n^{T} [A_{n,1} A_{n,1}^{T} + A_{n,2} A_{n,2}^{T} + A_{n-1,1}^{T} A_{n-1,1}
  + A_{n-1,2}^{T} A_{n-1,2}].
\end{align*}

Therefore,
\begin{align*} 
\dot{\textbf{A}}_n^{T} 
 = & [A_{n+1,1} A_{n+1,1}^{T} + A_{n+1,2} A_{n+1,2}^{T} + A_{n,1}^{T} A_{n,1} + A_{n,2}^{T} A_{n,2} ] \textbf{A}_n^{T}  \\
& - \textbf{A}_n^{T} [A_{n,1} A_{n,1}^{T} + A_{n,2} A_{n,2}^{T} + A_{n-1,1}^{T} A_{n-1,1}  + A_{n-1,2}^{T} A_{n-1,2}] \\
&  + \dot{H}_{n+1}^{-1/2} \textbf{E}_{n+1} H_{n}^{1/2} +
H_{n+1}^{-1/2} \textbf{E}_{n+1} \dot{H}_{n}^{1/2}. 
\end{align*}

From \eqref{At=HEH}, we get $\textbf{E}_{n+1} H_{n}^{1/2} = H_{n+1}^{1/2} \textbf{A}_n^T $ and $H_{n}^{-1/2} \textbf{E}_{n+1} = \textbf{A}_n^T H_{n}^{-1/2} $. Even, $H_{n}^{-1/2} \dot{H}_{n}^{1/2} = - \dot{H}_{n}^{-1/2} H_{n}^{1/2}$, and then
\begin{align} 
\dot{\textbf{A}}_n^{T} 
 = &  [A_{n+1,1} A_{n+1,1}^{T} + A_{n+1,2} A_{n+1,2}^{T}] \textbf{A}_n^{T}   - \textbf{A}_n^{T} [A_{n-1,1}^{T} A_{n-1,1} +  A_{n-1,2}^{T} A_{n-1,2}] \nonumber \\
&  + [ A_{n,1}^{T} A_{n,1} + A_{n,2}^{T} A_{n,2} + \dot{H}_{n+1}^{-1/2} H_{n+1}^{1/2} ] \textbf{A}^T_{n}   \label{LL_A} \\
& - \textbf{A}_n^T [A_{n,1} A_{n,1}^{T} + A_{n,2} A_{n,2}^{T} - \dot{H}_{n}^{-1/2} H_{n}^{1/2}]. \nonumber 
\end{align}

Relation \eqref{LL_A} can be seen as a formal type of 2D Langmuir lattice for the matrix coefficients of the three term relation of the orthonormal centrally symmetric polynomials.

\medskip

\noindent {\bf Acknowledgements.} The authors are grateful to the anonymous referee for the constructive comments, suggestions and to point out several references, that helped to improve the manuscript.


\end{document}